\documentclass[12pt, reqno]{amsart}
\usepackage{latexsym}
\usepackage{amssymb}
\usepackage{mathrsfs}
\usepackage{amsmath}
\usepackage{fancybox,color}
\usepackage{enumerate}
\usepackage{hyperref}
\usepackage[latin1]{inputenc}
\usepackage{eurosym}
\usepackage{verbatim}
\usepackage{color}

\newcommand{\kom}[1]{}
\renewcommand{\kom}[1]{{\bf [#1]}}

 \def\1{\raisebox{2pt}{\rm{$\chi$}}}
\def\a{{\bf a}}

\newtheorem{theorem}{Theorem}[section]
\newtheorem{corollary}[theorem]{Corollary}
\newtheorem{lemma}[theorem]{Lemma}
\newtheorem{proposition}[theorem]{Proposition}
\newtheorem{definition}[theorem]{Definition}

\newcommand{\R}{{\mathbb R}}

\newcommand{\N}{{\mathbb N}}

 \newcommand{\eps}{{\varepsilon}}
 \def\1{\raisebox{2pt}{\rm{$\chi$}}}
 
\newcommand{\abs}[1]{\left|#1\right|}
\newcommand{\norm}[1]{\left|\left|#1\right|\right|}
\newcommand{\Rn}{\mathbb{R}^n}
\newcommand{\diver}{\operatorname{div}}
\newcommand{\osc}{\operatorname{osc}}

\def\vint_#1{\mathchoice%
          {\mathop{\kern 0.2em\vrule width 0.6em height 0.69678ex depth -0.58065ex
                  \kern -0.8em \intop}\nolimits_{\kern -0.4em#1}}%
          {\mathop{\kern 0.1em\vrule width 0.5em height 0.69678ex depth -0.60387ex
                  \kern -0.6em \intop}\nolimits_{#1}}%
          {\mathop{\kern 0.1em\vrule width 0.5em height 0.69678ex
              depth -0.60387ex
                  \kern -0.6em \intop}\nolimits_{#1}}%
          {\mathop{\kern 0.1em\vrule width 0.5em height 0.69678ex depth -0.60387ex
                  \kern -0.6em \intop}\nolimits_{#1}}}
\def\vintslides_#1{\mathchoice%
          {\mathop{\kern 0.1em\vrule width 0.5em height 0.697ex depth -0.581ex
                  \kern -0.6em \intop}\nolimits_{\kern -0.4em#1}}%
          {\mathop{\kern 0.1em\vrule width 0.3em height 0.697ex depth -0.604ex
                  \kern -0.4em \intop}\nolimits_{#1}}%
          {\mathop{\kern 0.1em\vrule width 0.3em height 0.697ex depth -0.604ex
                  \kern -0.4em \intop}\nolimits_{#1}}%
          {\mathop{\kern 0.1em\vrule width 0.3em height 0.697ex depth -0.604ex
                  \kern -0.4em \intop}\nolimits_{#1}}}

\newcommand{\aveint}[2]{\mathchoice%
          {\mathop{\kern 0.2em\vrule width 0.6em height 0.69678ex depth -0.58065ex
                  \kern -0.8em \intop}\nolimits_{\kern -0.45em#1}^{#2}}%
          {\mathop{\kern 0.1em\vrule width 0.5em height 0.69678ex depth -0.60387ex
                  \kern -0.6em \intop}\nolimits_{#1}^{#2}}%
          {\mathop{\kern 0.1em\vrule width 0.5em height 0.69678ex depth -0.60387ex
                  \kern -0.6em \intop}\nolimits_{#1}^{#2}}%
          {\mathop{\kern 0.1em\vrule width 0.5em height 0.69678ex depth -0.60387ex
                  \kern -0.6em \intop}\nolimits_{#1}^{#2}}}

\newcommand{\ol}{\overline}
\newcommand{\Om}{\Omega}

\newcommand{\dist}{\operatorname{dist}}

\newcommand{\vp}{\varphi}
\numberwithin{equation}{section}

\allowdisplaybreaks

\newcommand{\diam}{\operatorname{diam}}

\newcommand{\tr}{\operatorname{tr}}
\renewcommand{\a}{\alpha}

\author[Attouchi]{Amal Attouchi}
\address{Department of Mathematics and Statistics, University of
	Jyv\"askyl\"a, PO~Box~35, FI-40014 Jyv\"askyl\"a, Finland}
\email{amalattouchi@gmail.com}

\author[Parviainen]{Mikko Parviainen}
\address{Department of Mathematics and Statistics, University of
	Jyv\"askyl\"a, PO~Box~35, FI-40014 Jyv\"askyl\"a, Finland}
\email{mikko.j.parviainen@jyu.fi}

\date{\today}
\begin{document}
	\title[Regularity for normalized $p$-Laplacian]{ H\"older regularity for the   gradient of the inhomogeneous parabolic  normalized $p$-Laplacian}
	
	\date{\today}
	\begin{abstract}
		In this paper we study an evolution equation involving the normalized $p$-Laplacian and a bounded  continuous source term. The normalized $p$-Laplacian is in non divergence form and arises for example from stochastic tug-of-war games with noise.
		 We prove local $C^{\alpha, \frac{\alpha}{2}}$ regularity for the spatial gradient of the viscosity solutions. The proof is based on an improvement of flatness and proceeds by iteration.
	\end{abstract}
\maketitle

\section{Introduction}
\label{sec:intro}

In this paper we are interested in  local regularity properties  for viscosity solutions of the normalized $p$-Laplace equation
\begin{equation}
\label{normpl}
\partial_t u(x,t)-\Delta_p^{N}u(x,t)=f(x,t),
\end{equation}
where $p\in (1, \infty)$ and $f$ is a continuous and bounded function.

The normalized $p$-Laplacian can be seen as the one-homogeneous version of the standard $p$-Laplacian and also as a combination of the Laplacian and the normalized infinity Laplacian,
\begin{align*}
\Delta_p^{N}u:&=|Du|^{2-p} \Delta_p u=\Delta u+(p-2)\Delta_{\infty}^N u\\
&=\Delta u+(p-2) |Du|^{-2}\sum_{ij} u_{ij}u_iu_j,
\end{align*}
where $\Delta_p u=\diver (|Du|^{p-2} Du)$,  $u_{ij}=\dfrac{\partial^2u}{\partial x_i\partial x_j}$ and $u_i=\dfrac{\partial u}{\partial x_i}$.

Recently,  a connection
between the theory of stochastic tug-of-war games and non-linear equations of
$p$-Laplacian type has been investigated.   In the elliptic case, this connection started with the seminal  work of Peres, Schramm, Sheffield and Wilson \cite{peress08, peresssw09}. In the parabolic case,  Manfredi, Parviainen and Rossi \cite{manfredipr10c}  showed that solutions  to   \eqref{normpl}  can be  obtained  as  limits  of  values  of  tug-of-war  games with noise  when
the parameter that controls the length of the possible movements goes to zero.

Equations of type \eqref{normpl} have been suggested in connection to  economics \cite{nystpar} and image processing \cite{does11,  moetz}. 

Existence of viscosity solutions of \eqref{normpl} has been proved using different techniques, including game-theoretic arguments \cite{manfredipr10c} and approximation methods (see \cite{bang13, bang15, does11}).  Regularity issues for this problem were analyzed in \cite{bang15,does11}  where the authors proved Lipschitz  estimates and bounds for the
modulus of continuity of \eqref{normpl} using PDE techniques.
In  \cite{parvruo} the authors obtained local H\"older and Lipschitz estimates using a  game theoretic method for the case $p=p(x,t)>2$.
The asymptotic behavior for \eqref{normpl} has been investigated in \cite{juutas, bang13, does11}.

The H\"older continuity of the solutions  follows from  the general regularity theory developed by Krylov and Safonov for equations in non-divergence form \cite{krylov_para_book,LWang1,cyril_luis}.
 Recently, for the homogeneous case $f=0$, Jin and Silvestre \cite{JinSil2015} proved  H\"older gradient estimates for solutions of \eqref{normpl}.  
They extended their method to prove H\"older gradient estimates for a class of singular or degenerate parabolic equations \cite{sil16}.

The normalized $p$-Laplacian enjoys the good  properties of being  uniformly parabolic and  1-homogeneous, the main difficulty in  proving regularity results comes from the discontinuity  at $\left\{ D u=0\right\}$.

Let $\Om$ be a bounded  domain of $\R^n$ and $T>0$. Denote $\Omega_T=\Omega\times (0, T)$. Our main contribution  is the following.
\begin{theorem}\label{thm:main1}
	Assume that $p>1$
	and $f\in L^{\infty}(\Omega_T)\cap C(\Omega_T)$.  There exists $\alpha=\alpha(p,n)>0$  such that any viscosity solution $u$ of
	\eqref{normpl} is in $C^{1+\alpha, \frac{1+\alpha}{2}}_{loc}(\Omega_T)$. Moreover, for any $\Om'\subset\subset\Om$ and $\eps>0$, we have
	\begin{equation}\label{fer}
	||Du||_{C^{\alpha, \frac\alpha2}(\Om'\times(\eps, T-\eps))} \le C \left(||u||_{L^\infty(\Om_T)} +
	\norm{f}_{L^\infty(\Om_T)} \right)
	\end{equation}
	and 
	\begin{equation}\label{timefer}
	\sup_{\substack{(x,t),(x,s)\in \Om'\times(\eps, T-\eps),\\
			t\neq s}}\,\frac{|u(x,t)-u(x,s)|}{|t-s|^{\frac{1+\alpha}{2}}}\leq C \left(||u||_{L^\infty(\Om_T)} +
	\norm{f}_{L^\infty(\Om_T)} \right),
	\end{equation}
	where $C=C(p,n,d,\eps,d')$, $d'=\dist(\Om', \partial\Om)$ and $d=\diam (\Om)$.
\end{theorem}

	The proof  is inspired by \cite{silimb}  and involves an improvement of flatness: if $u$ can be approximated by an affine function in a cylinder $Q$, then  we can  find a better approximation in a smaller cylinder and we can iterate the process. In fact, we prove by induction that for
	some $\rho, \alpha \in(0, 1)$ and $C=C(p,n)$, there exists a sequence $q_k$ such that 
	$ \underset{Q_{\rho^k}}{\osc} (u(x,t)-q_k\cdot x)\leq C \rho^{k(1+\alpha)}$. The inductive step is based on proving improvement of flatness  for the rescaled function $w_k=
	\rho^{-k(1+\alpha)}(u(\rho^k x) - q_k\cdot(\rho^k x))$.
	The main task  is then to study the equation satisfied by the deviation of $u$ from a  linear approximation that we denote $w(x,t)=u(x,t)-q\cdot x$ and to obtain
	a local $C^{\beta, \beta/2}$ estimate for $w$ independent of $q$. This will be the purpose of Lemma \ref{compactres}.
	By the Arzel\`a-Ascoli theorem we get compactness which, along with the regularity estimates for the homogeneous equation, we use to improve our approximation of $u$ in a smaller cylinder $Q_r$ by finding a  linear approximation for $w$ (see Lemma \ref{flatle}). Providing a linear approximation for $w$  independently of the value of $q$ is based on a contradiction argument  which requires a uniform $C^{1+\beta, \frac{\beta+1}{2}}$ estimate with respect to $q$ for the associated homogeneous equation. For this purpose, we adapt the strategy of \cite{JinSil2015} once we have obtained a Lipschitz control on $w$ independent of $q$ when $|q|$ is large enough.
	Finally, we prove \eqref{fer} by using Lemma \ref{flatle} and  an iteration in Lemma \ref{lemiter}.
	Once the  regularity with respect to the space  variable is obtained, we use a standard  barrier argument in order to control the oscillation of $u$ and derive  $C^{\frac{1+\alpha}{2}}$ regularity with respect to the time variable in Lemma \ref{timereg}.
	
	Let us also mention that the extremal cases $p=1$ and $p=\infty$ 
	have also received  attention. The case  $p=1$  is known as the  mean curvature flow equation, we refer the reader to the works of Evans and Spruck \cite{spruck}. The normalized infinity Laplacian is related to certain geometric problems and was studied by \cite{juuka,fliu}. We refer to \cite{evgame,kohnserf} for  game theoretic interpretations of these equations for the elliptic case.
	
	  We can adapt some tools from the  regularity theory of fully nonlinear parabolic equations  (ABP estimates, Harnack inequality, H\"older and Lipschitz regularity...). However the $C^{1+\alpha, (1+\alpha)/2}$ regularity results do not seem to be straightforwardly applicable to quasilinear equations  of the type \eqref{normpl} due to the discontinuity of the operator with respect to $Du$.
For classical results on regularity theory for equations in non-divergence form, we refer the reader to Krylov and Safonov \cite{krylov_para_book,krylov1980}, where they  used perturbation techniques in order to get $C^{1+\alpha, \frac{1+\alpha}{2}}$ regularity for fully nonlinear parabolic PDEs under a smallness condition on the solution. We also refer to the works    of L. Wang \cite{LWang1,wang_para_II} where he used compactness arguments in the case where the oscillation of the diffusion term is small enough. For a nice introduction to fully nonlinear parabolic equations, we refer to the  lecture notes of Imbert and Silvestre \cite{cyril_luis}.

	The paper is organized as follows. In Section 2 we present some preliminary tools, in Section 3 we prove Theorem \ref{thm:main1} and in the last section we provide estimates  for the homogeneous equation satisfied by the deviation of  $u$ from a linear function.

	\noindent \textbf{Acknowledgements.} 
	The second author was supported by the Academy of
Finland project \# 260791.
	\section{Preliminaries}

	For $r>0$, denote $Q_r:=B_r(0)\times (-r^2, 0]$ and $Q_r(x,t):= Q_r+(x,t)$ the parabolic cylinders.
	The normalized $p$-Laplacian can be seen as a uniformly parabolic operator in trace form on the set  $\left\{Du\neq 0 \right\}$. Indeed it can be written in the form
	$$\Delta_p ^N u=\tr (A(Du) D^2u)$$
	where 
	$$A(Du)=I+(p-2) \dfrac{Du\otimes Du}{|Du|^2},$$
	satisfies
	$$\lambda |\xi|^2\leq\left\langle  A(Du)\xi, \xi\right\rangle \leq \Lambda |\xi|^2.$$
	
	For $p>1$, we denote by $\Lambda$ and $\lambda$ the ellipticity constants of the normalized $p$-Laplacian $\Delta^N_p$. It is easy to  see that $\Lambda=\max (p-1, 1)$ and $\lambda=\min (p-1, 1)$.\\
	We denote by $\mathcal{S}^n$ the set of symmetric $n\times n$ matrices. For $a,b\in \R^n$, we denote by $a\otimes b$ the $n\times n$-matrix for which $(a\otimes b)_{ij}=a_i b_j$.\\
	We will use the \emph{Pucci operators}
	\[
	P^+(X):=\sup_{A\in \mathcal{A}_{\lambda,\Lambda}} -\tr(AX)
	\]
	and
	\[
	P^-(X):=\inf_{A\in \mathcal{A}_{\lambda,\Lambda}} -\tr(AX),
	\]
	where $\mathcal{A}_{\lambda,\Lambda}\subset \mathcal{S}^n$ is the  set of symmetric $n\times n$ matrices whose eigenvalues belong to $[\lambda,\Lambda]$.
	
Since we study  H\"older and $C^{1+\alpha, (1+\alpha)/2}$ regularity in parabolic cylinders $Q_r$, for $\alpha\in (0,1)$,  we will use the notation
	\[
	[u]_{C^{\alpha, \alpha/2}(Q_r)}:=\sup_{\substack{(x,t),(y,s)\in Q_r,\\
			(x,t)\neq (y,s)}} \dfrac{|u(x,t)-u(y,s)|}{|x-y|^\alpha+|t-s|^{\frac\alpha 2}},
	\]
	\[
	\norm{u}_{C^{\alpha, \alpha/2}(Q_r)}:=\norm{u}_{L^\infty(Q_r)}+[u]_{C^{\alpha, \alpha/2}(Q_r)}
	\]
	for H\"older continuous functions.

We define the space $C^{1+\alpha, (1+\alpha)/2}(Q_r)$  as the space of all functions with finite norm 
\[
	\norm{u}_{C^{1+\alpha, (1+\alpha)/2}(Q_r)}:=\norm{u}_{L^\infty(Q_r)}+\norm{Du}_{L^\infty(Q_r)}+[u]_{C^{1+\alpha, (1+\alpha)/2}(Q_r)},
	\]
	where
		\begin{align*}
[u]_{C^{1+\alpha,(1+\alpha)/2}(Q_r)}:&=\sup_{\substack{(x,t), (y,s)\in Q_r\\(x,t)\neq (y,s)}}\dfrac{|D u(x,t)-D u(y, s)|}{|x-y|^\alpha+|t-s|^{\frac\alpha 2}}\\
	&\quad+\sup_{\substack{(x,t),(x,s)\in Q_r,\\ t\neq s}}\,\frac{|u(x,t)-u(x,s)|}{|t-s|^{\frac{1+\alpha}{2}}}.
\end{align*}
	In this paper $C$ and $\tilde C$
will denote generic constants which may change from line to line.

	The normalized $p$-Laplacian is undefined when $D u=0$,
	where  it has  a bounded discontinuity. This difficulty can be resolved by  adapting
	the notion of viscosity solution using the upper and
	lower semicontinuous envelopes (relaxations) of the operator, see \cite{crandall1992user}.
	
	\begin{definition}
		Let $\Om$ be a bounded domain, $T>0$, $1<p<\infty$ and $f\in C(\Om_T)$. An upper semicontinuous function $u$ is  a viscosity subsolution of \eqref{normpl}
		if for all $(x_0,t_0)\in\Om_T$  and $\varphi\in C^2(\Om_T)$ such that $u-\varphi$ attains a local maximum at $(x_0, t_0)$, one has
		$$
		\left\{
		\begin{split}
		\varphi_t(x_0,t_0)-\Delta_p^N \varphi(x_0,t_0)\leq f(x_0,t_0),\, &\quad
		\text{if}\,\, D \varphi(x_0,t_0)\neq 0,\\
		\varphi_t(x_0,t_0)-\Delta\varphi(x_0,t_0)-(p-2)\lambda_{max}& (D^2\varphi(x_0,t_0))\leq f(x_0,t_0),\\
		&\quad\text{if}\,\, D\varphi(x_0,t_0)=0\,\,\text{and}\,\, p\geq 2,\\
		\varphi_t(x_0, t_0)-\Delta\varphi(x_0,t_0)-(p-2)\lambda_{min}& (D^2\varphi(x_0,t_0))\leq f(x_0,t_0),\\
		&\quad\text{if}\,\, D \varphi(x_0,t_0)=0, \, 1<p< 2.\\
		\end{split}\right.
		$$
		A lower semicontinuous function $u$ is a viscosity supersolution of \eqref{normpl}
		if for all $(x_0,t_0)\in\Om_T$  and $\varphi\in C^2(\Om_T)$ such that $u-\varphi$ attains a local minimum at $(x_0, t_0)$, one has
		
		$$	\left\{
		\begin{split}
		\varphi_t(x_0,t_0)-\Delta_p^N \varphi(x_0,t_0)\geq f(x_0,t_0),\, &\quad
		\text{if}\,\, D \varphi(x_0,t_0)\neq 0,\\
		\varphi_t(x_0,t_0)-\Delta\varphi(x_0,t_0)-(p-2)\lambda_{min}& (D^2\varphi(x_0,t_0))\geq f(x_0,t_0),\\
		&\quad\text{if}\,\, D\varphi(x_0,t_0)=0\,\,\text{and}\,\, p\geq 2,\\
		\varphi_t(x_0, t_0)-\Delta\varphi(x_0,t_0)-(p-2)\lambda_{max}& (D^2\varphi(x_0,t_0))\geq f(x_0,t_0),\\
		&\quad\text{if}\,\, D \varphi(x_0,t_0)=0, \, 1<p<2.
		\end{split}\right.$$
		We say that $u$ is a viscosity solution of \eqref{normpl} in $\Om_T$ if it is both a viscosity sub- and
		supersolution.
	\end{definition}

	In Sections 3 and 4 we will need the following lemma which allows to control the oscillation in space-time  for a solution  of a uniformly parabolic  equation with bounded discontinuities, once its oscillation in space is controlled in every time slice (see \cite[Lemma 4.3]{JinSil2015} and  \cite[Lemma 9.1]{barbito}).
	The viscosity solution to \eqref{modanot} is defined  analogously  to \eqref{normpl}.
	
	\begin{lemma}
		\label{lem:space2time}
		Let $v\in C(\ol Q_r)$  be a viscosity solution to 
		\begin{equation}\label{modanot}
		v_t-\Delta v-(p-2) \left\langle D^2v\frac{Dv +a}{\abs{Dv+a}}, \frac{Dv+a}{\abs{Dv+a}}\right\rangle=f\qquad\text{in}\quad Q_r,
		\end{equation}
		where $a\in \R^n$ and $f\in L^{\infty}(Q_r)\cap C(Q_r)$.
		Suppose that for all $t\in[-r^2, 0]$
		\[
		\begin{split}
		\underset{B_r}{\osc}\, v(\cdot, t)\le A,
		\end{split}
		\]
		then 
		\[
		\begin{split}
		\underset{Q_r}{\osc}\,v\le C(n,p)A+4r^2\norm{f}_{L^{\infty}(Q_r)}.
		\end{split}
		\]
	\end{lemma}
	\begin{proof}
		Denote $\Lambda=\max(1,p-1)$ and define $$\bar v(x,t):= \bar{\eta}+\left(2\norm{f}_{L^{\infty}(Q_r)}+\dfrac{5An\Lambda  }{r^2}\right)t+2Ar^{-2}|x|^2$$
		where $\bar{\eta}$ is chosen so that
		$\bar{v}(\cdot,-r^2)\geq v(\cdot,-r^2)$ and $\bar{v}(\bar{x},-r^2)=v(\bar{x}, -r^2)$
		for some $\bar x \in \bar B_r$. 
		Notice that $\bar{x}$ must belong to $B_r$  since otherwise  we would have
		\begin{align*}
		2A=\bar{v}(\bar x,-r^2)-\bar{v}(0,-r^2)&\leq v(\bar x,-r^2)-v(0,-r^2)\\&\leq \underset{B_r}{\osc} \,v(\cdot,-r^2)\leq A,
		\end{align*}
		which is impossible. Now we claim that $\bar{v}\geq v$ in $Q_r$. We argue by contradiction. If this is not true, set $m=-\underset{Q_r}{\inf} (\bar{v}-v)>0$. Let $(\hat x,\hat t)\in\overline{ Q_r}$ be such that $m=v(\hat x, \hat t)-\bar{v}(\hat x, \hat t )$. By assumption, we have that $\underset{B_r}{\osc}\, v(\cdot, \hat t)\leq A$. 
		Since by construction, $ \bar{v}(\bar x,-r^2)=v(\bar x, -r^2)$, we have $\hat t>-r^2$. Noticing that $\bar{v}+m\geq  v$ and $\bar{v}(\hat x, \hat t)+m=u(\hat x, \hat t)$, then arguing as above, we see that $\hat x\in B_r$. Since $v$ is a viscosity solution of \eqref{modanot} and $\bar v$ is a smooth  test function, 
		it follows by definition  that if $D\bar v(\hat x, \hat t)+a\neq 0$, we have
		\begin{align*}
		&\norm{f}_{L^{\infty}(Q_r)}+\dfrac{5An\Lambda}{r^2}\leq \partial_t \bar v(\hat x, \hat t)-f(\hat x, \hat t)\\
		&\qquad\quad\leq \Delta \bar v(\hat x, \hat t)+(p-2) \left\langle D^2 \bar v(\hat x , \hat t)\frac{D \bar v(\hat x, \hat t) +a}{\abs{D\bar v(\hat x,\hat t)+a}}, \frac{D \bar v(\hat x, \hat t)+a}{\abs{D\bar v(\hat x, \hat t)+a}}\right\rangle\\
		&\qquad\quad\leq \dfrac{4n\Lambda A}{r^2},
		\end{align*}
		and if  $D\bar v(\hat x, \hat t)+a= 0$, then
		\begin{align*}
		\norm{f}_{L^{\infty}(Q_r)}+\dfrac{5An\Lambda}{r^2}&\leq \partial_t \bar v(\hat x, \hat t)-f(\hat x, \hat t)\\
		&\leq \Delta \bar  v(\hat x, \hat t)+(p-2) \lambda_{\max}(D^2 \bar v(\hat x , \hat t))\\
		&\leq \dfrac{4n\Lambda A}{r^2},
		\end{align*}
		which is impossible.
		Using similar arguments, we can show that for some suitable $\underline{\eta}$ the function  $$\underline{v}(x,t):=\underline{\eta}- \left(2\norm{f}_{L^{\infty}(Q_r)}+\dfrac{5An\Lambda}{r^2}\right)t-2Ar^{-2}|x|^2$$ satisfies
		$\underline v(\cdot, -r^2)\leq v(\cdot,-r^2)$ and $\underline v(\underline x, -r^2)=v(\underline x,-r^2)$ for some $\underline x\in B_r$ and hence
		$$\underline{v}\leq v\in Q_r.$$
		Moreover, since $\bar v(\bar x, -r^2)-\underline v(\underline x, -r^2)\leq \underset{B_r}{\osc}\, v(\cdot, -r^2)\leq A$, we have 
		$$\bar\eta-\underline{\eta}\leq (10n \Lambda+1) A+4r^2\norm{f}_{L^{\infty}(Q_r)}$$
		and it follows that 
		$$\underset{Q_r}{\osc}\,  v\leq\underset{Q_r}{\sup}\, \bar v-\underset{Q_r}{\inf}\,  \underline{v}\leq \bar\eta-\underline{\eta}+4A\leq (10n\Lambda +5)A +4r^2\norm{f}_{L^\infty(Q_r)}.\qedhere$$	
	\end{proof}

	\section{Proof of Theorem \ref{thm:main1}}
	In this section we prove Theorem \ref{thm:main1}. We assume that $p>1$ and $f\in L^{\infty}(\Omega_T)\cap C(\Omega_T)$
	and want to show that there exists $\alpha=\alpha(p,n)>0$ such that any viscosity solution $u$ to
	equation \eqref{normpl} is of class $C^{1+\alpha. \frac{\alpha+1}{2}}_{loc}(\Omega_T)$.

	\subsection{Reduction of the problem and preliminary lemmas}
	First we reduce the problem by rescaling. For $\eps_0>0$,  let $$\kappa=(2\norm{u}_{L^{\infty}(\Om_T)}+\eps_0^{-1}\norm{f}_{L^{\infty}(\Om_T)})^{-1}.$$ Setting $\tilde{u}=\kappa u$, then $\tilde{u}$ satisfies
	\begin{equation*}
	\partial_t \tilde{u}-\Delta_p^N\left(\tilde{u}\right)=\tilde{f},
	\end{equation*}
	where $\tilde f:=\kappa f$ and $\norm{\tilde{u}}_{L^{\infty}(\Om_T)}\leq \frac{1}{2}$ and $\norm{\tilde{f}}_{L^{\infty}(\Om_T)}\le \eps_0$. Hence, without loss of generality we  may assume in Theorem \ref{thm:main1} that  $\norm{u}_{L^{\infty}(\Om_T)}\leq 1/2$ and $\norm{f}_{L^{\infty}(\Om_T)}\leq \eps_0$, where $\eps_0=\eps_0(p,n)$ is chosen later. Next, we notice that it is sufficient to prove that there exists some constant $C=C(p,n)>0$ such that for fixed $(x,t)\in\Om\times (0, T)$ and small enough $r\in(0,1)$, there exists $q=q(r,x,t)\in\R^n$ for which
	$$\underset{Q_r(x,t)}{\osc} \,(u(y,s)-u(x,t)-q\cdot(x- y))\leq Cr^{1+\a}.$$
	It suffices to choose a suitable $\rho\in(0,1)$ such that the previous inequality holds true for $r=r_k=\rho^k$,  $q=q_k$  and $C=1$ by proceeding by induction on $k\in\N$. Using a covering of  $\Omega_T$ by cubes  $Q_r(x,t)$ where $x\in \Omega$, $t\in (0,T)$ and $r<\text{dist}\, ((x,t),\partial \Omega_T)$, we may work on cubes $Q_r(x,t)$. By translation, it is enough to show that the solution is  $C^{1+\alpha, \frac{\alpha+1}{2}}$ at 0, and by considering 
	\[
	u_r(y, s)=r^{-2}u(x+ry, t+r^2s),
	\]
	we may work on the unit parabolic cylinder $Q_1(0,0)$. Finally, considering $u(x,t)-u(0,0)$ if necessary, we may suppose that $u(0,0)=0$.\\

	Notice that  if $w(x,t)=u(x,t)-q\cdot x$, then $w$ satisfies
	\begin{equation}\label{rescpb}
	\partial_t w-\Delta w-(p-2) \left\langle D^2w\frac{Dw+q}{\abs{Dw+q}}, \frac{Dw+q}{\abs{Dw+q}}\right\rangle=f\quad\text{in}\quad Q_1.
	\end{equation}
	In order to prove Theorem \ref{thm:main1}, we will first need the following  equi-continuity lemma.
	\begin{lemma}\label{compactres}
		For all $r\in (0,1)$, there exist $\beta=\beta(p,n)\in(0,1)$ and $C=C(p,n,r)>0$ such that any viscosity solution $w$ of \eqref{rescpb} with $\underset{Q_1}{\osc w}\leq 1$ and $\norm{f}_{L^{\infty}(Q_1)}\leq 1$ satisfies 
		\begin{equation}
		\norm{w}_{C^{\beta, \beta/2}( Q_r)}\leq C.
		\end{equation}
	\end{lemma}
	
	\begin{proof}
		Equation \eqref{rescpb}  can be rewritten as $$\partial_t w-\tr{\left( \left(I+(p-2)\dfrac{D w+q}{|D w+q|}\otimes\dfrac{D w+q}{|D w+q|}\right) D^2 w\right)}=f.$$
		Recalling the  definitions for Pucci operators $P^+$ and $P^-$ respectively, the equation takes the form
		$$\left\{\begin{array}{ll}\partial_tw +P^+(D^2w)+|f|\geq 0\\
		\partial_t w +P^-(D^2w)-|f|\leq 0\end{array}\right..$$
		By the classical results of  \cite{krylov1980,LWang1} (see also \cite{cyril_luis}), there exist $\beta=\beta(p,n)\in(0,1)$ and $C_\beta=C(r,p,n)$  such that
		\[
		\norm{w}_{C^{\beta, \beta/2}( Q_r)}\leq  C_\beta\left(\underset{Q_1}{\osc}\, w+\norm{f}_{L^{n+1}(Q_1)}\right)\leq C_\beta.\qedhere
		\] 
	\end{proof}
	
	Our next step consists in showing a linear approximation result for solutions to equation \eqref{rescpb}. We proceed by contradiction, using the previous proposition together with  known regularity results for uniformly parabolic linear  PDEs and the following  result for the homogeneous equation associated to \eqref{rescpb}.
	For convenience, we postpone the technical proof of Proposition \ref{regmodhom} and present it in the last  section.
	\begin{proposition}\label{regmodhom}
		Assume that $f\equiv0$ and let $w$ be a viscosity solution to equation \eqref{rescpb} with $\underset{Q_1}{\osc{w}}\leq 1$. For all $r\in (0,\frac34)$, there exist constants $ C_0=C_0(p,n)>0$  and $\beta_1=\beta_1(p,n)>0$ such that 
		\begin{equation}
		\begin{split}
		\norm{w}_{C ^{1+\beta_1, (1+\beta_1)/2}( Q_{r})}\leq  C_0.
		\end{split}
		\end{equation}\end{proposition}

	Now we are in position to prove the following improvement of flatness lemma.
	\begin{lemma}\label{flatle} There exist $\eps_0=\eps_0(p,n)\in(0,1)$ and $\rho=\rho_0(p,n)\in(0,1)$, such that for any $q\in\R^n$ and for any viscosity  solution $w$ of \eqref{rescpb} with $\osc_{Q_1}(w)\leq 1$ and $\norm{f}_{L^{\infty}(Q_1)}\leq \eps_0$, there exists $q'\in\Rn$ with $|q'|\leq \tilde C(p,n)$ such that
		$$\underset{Q_{\rho}}{\osc}\,(w(x,t)-q'\cdot x)\leq \frac{1}{2}\rho.$$
	\end{lemma}
	
	\begin{proof}
		Suppose by contradiction that there exist a sequence of functions $(f_j)$ with $\norm{f_j}_{L^\infty(Q_1)} \rightarrow 0$, a sequence of vectors $(q_j)$ and a sequence of viscosity solutions $(w_j)$ with $\osc_{Q_1}{w_j}\leq 1$ to
		\begin{equation*}
		\partial_t w_j-\Delta w_j-(p-2) \left\langle D^2w_j\frac{Dw_j+q_j}{\abs{Dw_j+q_j}}, \frac{Dw_j+q_j}{\abs{Dw_j+q_j}}\right\rangle=f_j\quad\text{in}\quad Q_1,
		\end{equation*}
		such that, for all $q'\in \Rn$ and any $\rho\in(0,1)$
		\begin{equation}\label{contra}
		\underset{Q_{\rho}}{\osc}(w_j(x,t)-q'\cdot x)>\frac{\rho}{2}.
		\end{equation}
		Using the Arzel\`{a}-Ascoli compactness result with  Lemma \ref{compactres}, there exists a continuous function $w_{\infty}$ such that $w_j\to w_{\infty}$ uniformly in $Q_\rho$ for any $\rho\in (0,1)$.
		Passing to the limit in \eqref{contra}, we have that for any vector $q'$,
		\begin{equation}\label{contrafin}
		\underset{Q_{\rho}}{\osc} (w_{\infty}(x,t)-q'\cdot x)>\dfrac{\rho}{2}.
		\end{equation}
		
		Let us first suppose that the sequence $(q_j)$ is bounded. Then, up to a subsequence, it converges to $q_{\infty}$. Using the relaxed limit, we get that   $w_{\infty}$  satisfies
		$$   \partial_t w_\infty-\Delta w_\infty-(p-2) \left\langle D^2w_\infty\frac{Dw_\infty+q_j}{\abs{Dw_\infty+q_\infty}}, \frac{Dw_\infty+q_\infty}{\abs{Dw_\infty+q_\infty}}\right\rangle=0,    $$
		applying  Proposition \ref{regmodhom} we have that,  there exists $C_0=C_0(p,n)>0$ such that
		$$\norm{w_{\infty}}_{C^{1+\beta_1, (1+\beta_1)/2}( Q_{1/2})}\leq C_0.$$
		If the sequence $(q_j)$ is unbounded, take a subsequence, still denoted by $(q_j)$, for which $|q_j|\rightarrow \infty$, and then a converging subsequence from $e_j=\frac{q_j}{|q_j|}$, $e_j\rightarrow e_\infty$. We have that in $Q_1$
		\begin{equation*}\partial_t w_j-\Delta w_j-(p-2) \left\langle  D^2w_j\frac{Dw_j|q_j|^{-1}+e_j}{\abs{Dw_j|q_j|^{-1}+e_j}}, \frac{Dw_j|q_j|^{-1}+e_j}{\abs{Dw_j|q_j|^{-1}+e_j}}\right\rangle=f_j.
		\end{equation*}
		Passing to the limit we obtain 
		\begin{equation}\label{der2}
		\partial_t w_\infty-\Delta w_{\infty}-(p-2) \left\langle D^2w_{\infty}\,e_{\infty}, e_{\infty}\right\rangle=0\qquad\text{in}\quad Q_1,
		\end{equation}
		with $|e_{\infty}|=1$. Noticing that equation \eqref{der2} can be written as
		\[
		\partial_t w_\infty-\tr{((I+(p-2) e_{\infty}\otimes e_{\infty}) D^2w_\infty)}=0,
		\]
		we see that equation \eqref{der2} is linear, uniformly parabolic and depends only on $ \partial_t w_\infty$ and $ D^2 w_\infty$. By the regularity result of \cite[Lemma 12.13]{lieberm}, there is $\beta_2>0$ so that $w_\infty\in C^{1+\beta_2,(1+\beta_2)/2}_\text{loc}$ and the H\"older norm of $Dw_\infty$ is bounded by a constant depending only on $p,n,\norm{w_\infty}_{L^\infty(Q_1)}$ that we still denote by $C_0$.
		
		We have thus shown that $w_\infty\in C^{1+\beta,(1+\beta)/2}_\text{loc}$ for $\beta=\min(\beta_1,\beta_2)>0$ with a H\"older norm independent of the sequence $(q_j)_j$. Choose $\rho\in (0, 1/2)$ such that
		\begin{equation*}
		C_0\rho^{\beta}\leq \frac{1}{4}.
		\end{equation*}
		By $C^{1+\beta,(1+\beta)/2}_\text{loc}$ regularity there exists a vector $k_{\rho}$  with $\norm{k_\rho}\leq \tilde C(p,n)$  (see Proposition \ref{regmodhom}) such that 
		\begin{equation}\label{kilopr}
		\underset{Q_{\rho}}{\osc} (w_{\infty}(x,t)-k_{\rho}\cdot x)\leq C_0\rho^{1+\beta}\leq \frac{1}{4}\rho.
		\end{equation}
		This contradicts \eqref{contrafin}, and the proof is complete.
		The boundedness of $|q'|$  can be obtained as follows.
		Since $w_j$ converges to $w_\infty$, we get that for all $\eps_1>0$ there exists   $\eps_0=\eps_0(p,n)$ sufficiently small such that  if $\norm{f}_{L^\infty (Q_1)}\leq \eps_0$ then
		\begin{equation}\label{kiju}
		\underset{Q_\rho}{\osc}(w-w_\infty)\leq \eps_1.
		\end{equation}
Taking $\eps_1=\frac{\rho}{4}$, we conclude from \eqref{kilopr} and \eqref{kiju}
that for $\eps_0$ sufficiently small, if $\norm{f}_{L^\infty (Q_1)}\leq \eps_0$ then
\begin{equation*}\underset{Q_\rho}{\osc}(w(x, t)-k_\rho\cdot x)\leq \frac{\rho}{2}	.\qedhere\end{equation*}
	\end{proof}
	
	\subsection{Iteration and proof of the main theorem}
$\,$\\
\textbf{H\"older estimate for the gradient in  the space variable}\\
	In order to control the H\"older  continuity of the gradient of a function with respect to the space variable, it is standard to make sure that, around each point, the function can be deviated  from a plane so that its  oscillation in a ball of radius $r>0$ is of order $r^{1+\alpha}$ (see \cite[Lemma 12.12]{lieberm} and Appendix \ref{appen1}).
	
	The H\"older regularity with respect to the space variable stated in  Theorem \ref{thm:main1} is a direct consequence of  the following lemma  and Lemma \ref{lemendi} after scaling back from $\tilde{u}$ to $u$. 
	\begin{lemma}\label{lemiter}
		Assume  that $\rho$, $\eps_0\in(0,1)$ are as in Lemma \ref{flatle} and let $u$ be a viscosity solution of \eqref{normpl} with $\underset{Q_1}{osc}\, u\leq1$ and $\norm{f}_{L^\infty(Q_1)}\leq \eps_0$.
		Then there exists $\a\in(0,1)$ such that, for all $k\in\N$, there exists $q_k\in \R^n$ such that
		\begin{equation}\label{itera}
		\underset{Q_{r_{k}}}{\osc} \, (u(y,t)-q_{k}\cdot y)\leq  r_k^{1+\a },
		\end{equation}
		where $r_k:=\rho^k$.
	\end{lemma}
	\begin{proof}
		For $k=0$, the estimate \eqref{itera} follows from the assumption $\underset{Q_1}{osc}\, u\leq 1$ and $q_0=0$. Next we take $\alpha\in (0,1)$ such that $\rho^{\alpha}>1/2$.
		We assume that $k\geq 0$ and that we constructed already $q_k\in \R^n$ such that \eqref{itera} holds true. 
		To prove the inductive step $k\rightarrow k+1$, we denote $r_k:=\rho^k$ and we rescale the solution considering for $x\in Q_1$
		$$w_k(x,t)=r_k^{-1-\alpha}\big(u(r_k x, r_k^2t)-q_k\cdot (r_k x)\big).$$
		We have by induction assumption, $\underset{Q_1}{\osc}\,(w_k)\leq 1$ and $w_k$ satisfies  
		$$\partial_t w_k-\Delta w_k-(p-2) \left\langle D^2w_k\frac{Dw_k+(q_k/r_k^{\alpha})}{\abs{  Dw_k+(q_k/r_k^{\alpha})}}, \frac{Dw_k+(q_k/r_k^{\alpha})}{\abs{Dw_k+(q_k/r_k^{\alpha})}}\right\rangle=f_k,  $$
		where  $f_k(x,t)=r_k^{1-\alpha}f(r_k x, r_k^2t)$ with $\norm{f_k}_{L^{\infty}(Q_1)}\leq \eps_0$ since $\alpha<1$.
		Using the result of Lemma \ref{flatle}, there exists $l_k\in\R^n$ with $|l_k|\leq \tilde C(p,n)$ such that
		$$\underset{B_{\rho}}{\osc}\,(w_k(x,t)-l_k\cdot x)\leq \frac{1}{2}\rho.$$   
		Setting 
		\begin{equation}\label{merci}
		q_{k+1}=q_k+ l_k r_k^{\a},
		\end{equation}
		 we get
		\[
		\underset{Q_{r_{k+1}}}{\osc} \, (u(x,t)-q_{k+1}\cdot x)\leq \dfrac{\rho}{2} r_k^{1+\a }\leq r_{k+1}^{1+\a}.\qedhere
		\]
	\end{proof}

	\proof[ Proof of estimate \eqref{fer} in  Theorem \ref{thm:main1}:]
	 We show that $q_k$ converges to a vector  $q_\infty$.
	Indeed from \eqref{merci}, we have that for $m\geq k$,
	$|q_m-q_k|\leq C\overset{m-1}{\underset{j=k}{\sum}} r_j^\alpha\leq C \rho^{k\a}$,  where $C=C(p,n)>0$. It follows that $q_k$ converges and 
	$$\sup_{(x,t)\in Q_{r_k}}(q_k\cdot x-q_\infty\cdot x)\leq C \rho^{k(1+\alpha)}, \quad\underset{(x,t)\in Q_{r_k}}{\osc}(u(x,t)-q_k\cdot x)\leq \rho^{k(1+\alpha)}.$$
	 Consequently, we have that 
	 \begin{equation}\label{aminater}
	\sup_{(x,t)\in Q_{r^k}} |u(x,t)-q_\infty \cdot x-u(0,0)|\leq C r_k^{1+\alpha},
	 \end{equation} where $C=C(p,n)$.
	 It follows from Lemma \ref{lemendi} that $Du$ is of class $C^{\alpha, \alpha/2}$ at $(0, 0)$ and we will denote $q_\infty$ by  $Du(0,0)$ in the sequel. The estimate \eqref{fer} in Theorem \ref{thm:main1} follows from translation arguments, estimate \eqref{aminater} and Lemma \ref{lemendi}.\qedhere

	\noindent\textbf{H\"older estimate for the solution in the time variable}
	$\,\,$\\
	Here we show the regularity of $u$  with respect to the time variable to finish the proof of Theorem \ref{thm:main1} by constructing suitable barriers in order to control the oscillation of $u$ (see \cite[Lemma 4.3]{JinSil2015}).
	\begin{lemma}\label{timereg}
	Under the hypothesis of Lemma \ref{lemiter}, there exists 
	$C=C(p,n)>0$ such that for all $t\in (-r^2,0)$, we have
	\begin{equation}
	|u(0, t)-u(0,0)|\leq C|t|^{\frac{1+\alpha}{2}}.
	\end{equation}
	\end{lemma}
	
	\begin{proof}
  For $(x,t)\in Q_r$, set
	$$v(x,t):=u(x,t)-u(0,0)-Du(0,0)\cdot x.$$
It follows from \eqref{itera} that for $x_1, x_2\in B_r$ and $t\in [-r^2, 0]$, we have
\begin{equation*}
	|v(x_1,t)-v(x_2, t)|\leq 
	\underset{(y,s)\in Q_r}{\osc}\, (u(y,s)-Du(0,0)\cdot y)\leq C r^{1+\a}.
\end{equation*}
We conclude that, for all $t\in [-r^2,0]$, we have
	$$\underset{B_r}{\osc}\, v(\cdot,t)\leq Cr^{1+\alpha}=:A.$$
	We claim that there exists a constant $C=C(p,n)>0$ such that 
	$$\underset{Q_r}{\osc}\, v\leq CA+4r^2\norm{f}_{L^\infty(Q_1)}.$$
	Indeed, denoting $b=Du(0,0)$, we
	observe that $v$ satisfies in $Q_r$ the following equation  in the viscosity sense 
	\begin{equation*}
	\partial_t v-\Delta v-(p-2)\left\langle D^2v \dfrac{Dv+b}{|Dv+b|},\dfrac{Dv+b}{|Dv+b|} \right\rangle=f.
	\end{equation*}
	We conclude from Lemma \ref{lem:space2time} that 
	\begin{align*}
	\underset{Q_r}{\osc}\, v&\leq C(p,n)r^{1+\alpha}+4r^2\norm{f}_{L^\infty(Q_1)}\\
	&\leq C(p,n)r^{1+\alpha}.
	\end{align*}
	In particular  for $t\in (-r^2, 0)$, we have
	$$|u(0,t)-u(0,0)|=|v(0,t)|\leq C(p,n)|t|^{\frac{1+\alpha}{2}}.      $$
	Since the equation \eqref{normpl} is invariant under translation, we get the desired result and have thus proven the second estimate \eqref{timefer} in Theorem \ref{thm:main1}.
	\end{proof}
	\section{Local H\"older estimate for the gradient of the limiting equation}
	In this section we derive  estimates for bounded viscosity solutions to the following equation
	\begin{equation}\label{modhom}
	w_t-\Delta w-(p-2)\left\langle D^2w\dfrac{Dw+q}{|Dw+q|},\dfrac{Dw+q}{|Dw+q|}\right\rangle  =0   \quad\text{in}\,\, Q_1,
	\end{equation}
	and in particular prove Proposition \ref{regmodhom}. 
	
\noindent	Introducing the function $v(x,t):=w(x,t)+q\cdot x$, it is easy to check that $v$ is  a viscosity solution to 
	\[v_t-\Delta_p^N v=0\qquad\text{in}\,\, Q_1.\]
	By the  regularity result of Jin and Silvestre \cite[Theorem 1.1]{JinSil2015}, there is $\beta_1=\beta_1(p,n)>0$ so that $v\in C^{1+\beta_1,(1+\beta_1)/2}_{\text{loc}}(Q_1)$ and hence  also $w\in C^{1+\beta_1,(1+\beta_1)/2}_{loc}(Q_1)$. 
	The main difficulty is to provide  $C^{1+\beta_1, (1+ \beta_1)/2}$ estimates which are uniform with respect to $q$. 
	
	The main idea is  to divide the study to the   cases $|q|$ small and $|q|$ large (depending on $p$ and $n$).
	For $|q|$ large enough, that is $|q|>L_0(p,n)$ where $L_0(p,n)$ will be  chosen later, our strategy is to prove  that the  Lipschitz norm of $w$ is controlled independently from $q$, hence the equation is no longer  discontinuous and we can adapt the proof of \cite{JinSil2015} to derive  $C^{1+\beta_1,(1+\beta_1)/2}$ estimates uniform with respect to $q$.  
	When $|q|\leq L_0$ the result follows immediately from \cite[Theorem 1.1]{JinSil2015}. Indeed,  since 
	\[\underset{Q_1}{\osc}\, v\leq \underset{Q_1}{\osc w}+ 2|q|\leq 1+2L_0,\]
	we get that 
	\begin{align*}
	\norm{w}_{C^{1+\beta_1, (1+\beta_1)/2}(Q_{1/2})}&\leq \norm{v}_{C^{1+\beta_1, (1+\beta_1)/2}(Q_{1/2})}+2|q|\\
	&\leq C(p,n)\underset{Q_1}{\osc v}+2L_0\leq C_0(p,n).
	\end{align*}
	Hence, from now on we will focus on the case $|q|$ large.
	\subsection{Lipschitz through the Ishii-Lions method}

	In order to prove Proposition \ref{regmodhom},   we first  need the following technical lemma concerning Lipschitz regularity of solutions of equation \eqref{modhom}. For $n\times n$ matrices we use the matrix norm
	\[
	||A||:=\sup_{|x|\leq 1}\{|Ax|\}.
	\] 
	
	\begin{lemma}\label{liphom1}
		Let $w$ be a bounded viscosity solution to equation \eqref{modhom} with $\underset{Q_1}{\osc}\, w\leq 1$. For all $r\in \left(0,\frac34\right)$, there exists a constant $ \nu_0=\nu_0(p,n)>0$ such that, if $|q|>\nu_0$, then for all  $x,y\in \overline{B_{r}}$ and $t\in [-r^2,0]$,
		\begin{equation}
		\begin{split}
		\abs{w(x,t)-w(y,t)}\le \tilde C\abs{x-y}, 
		\end{split}
		\end{equation}
		where  $\tilde{C}=\tilde{C}(p,n)>0$.
	\end{lemma}
	\begin{proof}
		We use the viscosity method introduced by Ishii and Lions in \cite{ishiilions} (see also \cite{cyril_luis, sil16} for further applications).
		\subsection*{Step 1}
		 
Notice that it suffices to show that $w$ is Lipschitz in $Q_{3/4}$, this will imply that $w$ is Lipschitz in any smaller cube $Q_r$ for $r\in \left(0,\frac34\right)$  with the same Lipschitz constant.	In the sequel we take  $r=3/4$.	
		First we fix $x_0, y_0\in B_{r}$, $t_0\in (-r^2,0)$ and introduce the auxiliary function 
		\begin{align*}
		\Phi(x, y,t):&=w(x,t)-w(y,t)-L\vp(\abs{x-y})\\
		&\quad-\frac M2\abs{x-x_0}^2-\frac M2\abs{y-y_0}^2-\frac M 2 (t-t_0)^2,
		\end{align*}
		where $\vp$ is defined below. Our aim is to show that $\Phi(x, y,t)\leq 0$ for $(x,y)\in \overline{B_r}\times\overline{ B_r}$ and $t\in [-r^2,0]$.
		For a proper choice of $\vp$, this yields the desired regularity result.
		We take
		\[
		\begin{split}
		\vp(s)=
		\begin{cases}
		s-s^{\gamma}\kappa_0& 0\le t\le s_1:=(\frac 1 {\gamma\kappa_0})^{1/(\gamma-1)}  \\
		\vp(s_1)& \text{otherwise},
		\end{cases}
		\end{split}
		\]
		where $2>\gamma>1$ and $\kappa_0>0$ is such that  $s_1\geq 2 $ and $\gamma \kappa_02^{\gamma-1}\leq 1/4$. 
		
		Then
		\[
		\begin{split}
		\vp'(s)&=\begin{cases}
		1-\gamma s^{\gamma-1}\kappa_0 & 0\le s\le s_1 \\
		0& \text{otherwise},
		\end{cases}\\
		\vp''(s)&=\begin{cases}
		-\gamma(\gamma-1)s^{\gamma-2} \kappa_0& 0<s\le s_1 \\
		0 & \text{otherwise}.
		\end{cases}
		\end{split}
		\]
		In particular, $\varphi'(s)\in  [\frac34,1]$  and $\varphi''(s)<0$ when $s\in [0,2]$.
		\subsection*{Step 2} We argue by contradiction and assume that
		$\Phi$ has a positive
		maximum at some point $(x_1, y_1,t_1)\in \bar B_r\times \bar B_r\times [-r^2,0]$.
		Notice  that $x_1\neq y_1$, otherwise the maximum of $\Phi$ would be non positive.
		Since  $w$ is  continuous and
		its oscillation is bounded, choosing $$M\geq \dfrac{32\osc_{Q_1}{w}}{\max(d\left( (x_0,t_0) ,\partial Q_r\right), d\left( (y_0,t_0) ,\partial Q_r\right))^2},$$ we get
		\begin{equation*} 
\begin{split}
		\abs{x_1-x_0}+|t_1-t_0|&\leq 2 \sqrt{\dfrac{2|w(x_1,t_1)-w(y_1,t_1|)}{M}}\leq 2\sqrt{ \frac{2\osc_{Q_1}{w}}{M}}\\
		&\leq \dfrac{d\left( (x_0,t_0),  \partial Q_r\right)}{2},\\
		\abs{y_1-y_0}+|t_1-t_0|&\leq 2\sqrt{\dfrac{2|w(x_1,t_1)-w(y_1,t_1|)}{M}}\leq 2\sqrt{\frac{2\osc_{Q_1}{w}}{M}}\\
		&\leq \dfrac{d\left( (y_0,t_0),  \partial Q_r\right)}{2},
		\end{split}
		\end{equation*}
		so that $x_1$ and $y_1$ are in $B_r$ and $t_1\in (-r^2,0)$.
		
\noindent		We know by Lemma \ref{compactres} that,   $w$ is locally H\"older continuous and that  there exists a constant $C_{\beta}>0$ depending only on $p, n$ and $\underset{Q_1}{\osc}\, w$ such that 
		$$|w(x,t)-w(y,t )|\leq C_\beta|x-y|^\beta \quad\text{for}\, x, y\in B_r, t\in (-r^2,0).$$	
		Using that $w$ is H\"older continuous, it follows, adjusting the constants (by  choosing $2M\leq C_{\beta}$), that  
		\begin{equation*}
		\begin{split}
		M\abs{x_1-x_0}\leq C_{\beta}&\abs{x_1-y_1}^{\beta/2},\\
		M\abs{y_1-y_0}\leq C_{\beta}&\abs{x_1-y_1}^{\beta/2}.
		\end{split}
		\end{equation*}

		By  Jensen-Ishii's lemma (also known as theorem of sums, see \cite[Theorem 8.3]{crandall1992user} and \cite{cyril_luis}), there exist
		\[
		\begin{split}
		&(\sigma,\tilde \zeta_x,X)\in \ol P^{2,+}\left(w(x_1,t_1)-\frac M2\abs{x_1-x_0}^2-\frac M2(t_1-t_0)^2\right),\\
		&(\sigma, \tilde \zeta_y,Y)\in \ol P^{2,-}\left(w(y_1,t_1)+\frac M2\abs{y_1-y_0}^2\right),
		\end{split}
		\]
		that is 
		\[
		\begin{split}
		&(\sigma+M(t_1-t_0),a,X+MI)\in \ol P^{2,+}w(x_1,t_1),\\ &(\sigma,b,Y-MI)\in \ol P^{2,-}w(y_1,t_1),
		\end{split}
		\]
		where ($\tilde \zeta_x=\tilde \zeta_y$)
		\[
		\begin{split}
		a&=L\vp'(|x_1-y_1|) \frac{x_1-y_1}{\abs{x_1-y_1}}+M(x_1-x_0)=\tilde \zeta_x+M(x_1-x_0),\\
		b&=L\vp'(|x_1-y_1|) \frac{x_1-y_1}{\abs{x_1-y_1}}-M(y_1-y_0)=\tilde \zeta_y-M(y_1-y_0).
		\end{split}
		\]
		If $L$ is large enough (depending on the H\"older constant  $C_\beta$), we have
		\[
		\abs{a},\abs{b}\geq L\varphi'(|x_1-y_1|) - C_\beta\abs{x_1-y_1}^{\beta/2}\ge \frac L2.
		\]
\noindent Moreover, by  Jensen-Ishii's lemma, for any $\tau>0$, we can take $X, Y\in \mathcal{S}^n$ such that 
		\begin{equation}\label{matriceineq1}
		- \big[\tau+2\norm{B}\big] \begin{pmatrix}
		I&0\\
		0&I 
		\end{pmatrix}\leq
		\begin{pmatrix}
		X&0\\
		0&-Y 
		\end{pmatrix}
		\end{equation}
		and
		\begin{equation}\label{matineq2}
			\begin{pmatrix}
		X&0\\
		0&-Y 
		\end{pmatrix}
		\le 
		\begin{pmatrix}
		B&-B\\
		-B&B 
		\end{pmatrix}
		+\frac2\tau \begin{pmatrix}
		B^2&-B^2\\
		-B^2&B^2 
		\end{pmatrix},
		\end{equation}
		where 
		\begin{align*}	
		B=&L\vp''(|x_1-y_1|) \frac{x_1-y_1}{\abs{x_1-y_1}}\otimes \frac{x_1-y_1}{\abs{x_1-y_1}}\\
		&\quad +\frac{L\vp'(|x_1-y_1|)}{\abs{x_1-y_1}}\Bigg( I- \frac{x_1-y_1}{\abs{x_1-y_1}}\otimes \frac{x_1-y_1}{\abs{x_1-y_1}}\Bigg)
		\end{align*}	
		and 
		\begin{align*}	
		B^2=
		&\frac{L^2(\vp'(|x_1-y_1|))^2}{\abs{x_1-y_1}^2}\Bigg( I- \frac{x_1-y_1}{\abs{x_1-y_1}}\otimes \frac{x_1-y_1}{\abs{x_1-y_1}}\Bigg)\\
		&\quad +L^2(\vp''(|x_1-y_1|))^2 \frac{x_1-y_1}{\abs{x_1-y_1}}\otimes \frac{x_1-y_1}{\abs{x_1-y_1}}.
\end{align*}	
		Notice that 
		\begin{equation}\label{lilou}
		\norm{B}\leq L \vp'(|x_1-y_1|),
		\end{equation}
		\begin{equation}\label{filou}
	 \norm{B^2}\leq L^2\left(|\vp''(|x_1-y_1|)|+\dfrac{|\vp'(|x_1-y_1|)|}{|x_1-y_1|}\right)^2,
		\end{equation}
		and for $\xi=\frac{x_1-y_1}{\abs{x_1-y_1}}$, we have
	
				\begin{equation*}
		\langle B\xi,\xi\rangle=L\vp''(|x_1-y_1|)<0, \qquad\langle B^2\xi,\xi\rangle=L^2(\vp''(|x_1-y_1|))^2.
		\end{equation*}
		Choosing $\tau=4L\left(|\vp''(|x_1-y_1|)|+\dfrac{|\vp'(|x_1-y_1|)|}{|x_1-y_1|}\right)$,  we have that for $\xi=\frac{x_1-y_1}{\abs{x_1-y_1}}$,
			\begin{align}\label{mercit}
		\langle B\xi,\xi\rangle +\frac2\tau \langle B^2\xi,\xi\rangle&=L\left(\vp''(|x_1-y_1|)+\frac2\tau L(\vp''(|x_1-y_1|))^2\right)\nonumber\\
		&\leq \dfrac{L}{2}\vp''(|x_1-y_1|)<0 .
		\end{align}
		In particular applying  inequalities \eqref{matriceineq1} and \eqref{matineq2} to any  vector $(\xi,\xi)$ with $\abs{\xi}=1$, we  have that $X- Y\leq 0$ and $\norm{X},\norm{Y}\leq 2\norm{B}+\tau$.  We refer the reader to \cite{ishiilions,crandall1992user} for details.\\
		Thus, setting $\eta_1=a+q$, $\eta_2=b+q$, we have for $\abs{q}$ large enough (depending only on $L$)
		\begin{align}\label{koivu}
		\abs{\eta_1}&\geq \abs{q}-\abs{a}\geq \frac{\abs{a}}{2}\geq \frac L4,\nonumber\\
		\abs{\eta_2}&\geq \abs{q}-\abs{b}\geq \frac{\abs{b}}{2}\geq \frac L4,
		\end{align}
		where $L$ will be chosen later on and $L$ will depend only on $p,n, C_\beta$.
		Writing the  viscosity inequalities
		\begin{equation*}
		\begin{split}
		M(t_1-t_0)+\sigma&\leq  \tr (X+MI)+(p-2)\dfrac{\left\langle(X+MI) (a+q), (a+q)\right\rangle}{|a+q|^2}\\
		\sigma&\geq  \tr(Y-MI)+(p-2) \dfrac{\left\langle(Y-MI) (b+q), (b+q)\right\rangle}{|b+q|^2},
		\end{split}
		\end{equation*}
	 we get 
		\begin{equation*}
		\begin{split}
		M(t_1-t_0)+\sigma&\leq  \tr (A(\eta_1)(X+MI))\\
		-\sigma&\leq  -\tr (A(\eta_2)(Y-MI))
		\end{split}
		\end{equation*}
		where  for $\eta \neq 0$ $\bar \eta=\dfrac{\eta}{|\eta |}$ and 
		\[A(\eta):= I+(p-2)\ol\eta\otimes \ol\eta.\]
		Adding the two inequalities, we get 
		\begin{equation*}
		0 \leq  \tr (A(\eta_1)(X+MI))
		-\tr (A(\eta_2)(Y-MI))+M|t_1-t_0|.
		\end{equation*}
		It follows that 
		\begin{align}\label{gregory}
		0 \leq  &\tr (A(\eta_1)(X-Y))
		+tr ((A(\eta_1)-A(\eta_2))Y)\nonumber\\
		&+M\big[\tr (A(\eta_1))+\tr (A(\eta_2)) \big]+2Mr^2.
		\end{align}
		
		Notice that all the eigenvalues of $X-Y$ are non positive. Moreover,  applying the previous matrix inequality \eqref{matineq2} to the vector $(\xi,-\xi)$ where $\xi:=\frac{x_1-y_1}{|x_1-y_1|}$  and using \eqref{mercit}, 
		we obtain
		\begin{align}\label{camille}
		\langle (X-Y) \xi, \xi\rangle&\leq 4\left(\langle B\xi,\xi\rangle+\frac2\tau\langle B^2\xi,\xi\rangle)\right)\nonumber \\
		&\leq 2 L\vp''(|x_1-y_1|)<0.
		\end{align}
		Hence  at least one of the eigenvalue of $X-Y$  that we denote by  $\lambda_{i_0}$ is   negative and smaller than $2 L\vp''(|x_1-y_1|)$. The eigenvalues of $A(\eta_1)$ belong to $[\min(1, p-1), \max(1, p-1)]$.	Using \eqref{camille}, it follows by \cite{theo} that 
		\begin{align*}  
		\tr(A(\eta_1) (X-Y))&\leq \sum_i \lambda_i(A(\eta_1))\lambda_i(X-Y)\\
		&\leq \min(1, p-1)\lambda_{i_0}(X-Y)\\
		&\leq 2\min(1, p-1) L \vp''(|x_1-y_1|).
		\end{align*}
		It is easy to see that 
		\begin{align*}A(\eta_1)-A(\eta_2)&=\ol\eta_1\otimes \ol\eta_1-\ol\eta_2\otimes \ol\eta_2\\
	&=(\ol\eta_1-\ol\eta_2+\ol\eta_2)\otimes\ol\eta_1-\ol\eta_2\otimes(\ol\eta_2-\ol\eta_1+
	\ol\eta_1)\\
	&=(\ol\eta_1-\ol\eta_2)\otimes\ol\eta_1+\ol\eta_2\otimes\ol\eta_1
	-\ol\eta_2\otimes(\ol\eta_2-\ol\eta_1)-\ol\eta_2\otimes\ol\eta_1\\
	&=(\ol\eta_1-\ol\eta_2)\otimes\ol\eta_1
	-\ol\eta_2\otimes(\ol\eta_2-\ol\eta_1)
		\end{align*}
		and hence
		\begin{align*}
		\tr( (A(\eta_1)-A(\eta_2)) Y)&\leq n\norm{Y}
		\norm{A(\eta_1)-A(\eta_2)}  \\
		&\leq n\abs{p-2}\norm{Y}|\ol\eta_1-\ol\eta_2|\left( |\ol\eta_1|+|\ol\eta_2|\right)\\
		&\leq 2n\abs{p-2}\norm{Y}|\ol \eta_1-\ol\eta_2|.
		\end{align*}
		On one hand we have
		\begin{equation*}
		\begin{split}
		\abs{\ol \eta_1-\ol \eta_2}&=
		\abs{\frac{\eta_1}{\abs {\eta_1}}-\frac{\eta_2}{\abs {\eta_2}}}
\le \max\left( \frac{\abs{\eta_2- \eta_1}}{\abs{\eta_2}},\frac{ \abs{\eta_2- \eta_1}}{\abs{\eta_1}}\right)\\
		&\le \frac {8C_\beta}{ L}\abs{x_1-y_1}^{\beta/2},
		\end{split}
		\end{equation*}
		where we used \eqref{koivu}.\\
		On the other hand, by \eqref{matriceineq1}--\eqref{filou},
		\begin{align*}
		\norm{Y}&=\max_{\ol \xi} |\langle Y\ol \xi, \ol \xi\rangle|
		\le 2 |\langle B\ol \xi,\ol \xi \rangle|+\frac4\tau|\langle B^2\ol \xi,\ol \xi \rangle| \\
		&\leq 4L\left( \frac{\vp'(|x_1-y_1|)}{\abs{x_1-y_1}}+ |\vp''(|x_1-y_1|)|\right).
		\end{align*}
		Hence, remembering that $|x_1-y_1|\leq 2$, we end up with
	\begin{align*}
	 \tr( (A(\eta_1)-A(\eta_2)) Y)&\leq 128n\abs{p-2}C_\beta \vp'(|x_1-y_1|) \abs{x_1-y_1}^{-1+\beta/2}\\
	 &\quad+128n\abs{p-2}C_\beta |\vp''(|x_1-y_1|)|.
	 \end{align*}
		 Finally, we have
		$$ M(\tr(A(\eta_1))+\tr(A(\eta_2)))\leq 2Mn\max(1, p-1).$$

		\subsection*{Step 3}
		
		Gathering the previous estimates with \eqref{gregory} and recalling the definition of $\vp$, we get 
		\begin{align*}
		0&\leq 128n\abs{p-2}C_\beta\left(\vp'(|x_1-y_1|)  \abs{x_1-y_1}^{\beta/2-1}+  |\vp''(|x_1-y_1|)|\right)\\
		&\quad+2\min(1, p-1) L \vp''(|x_1-y_1|)  +2Mr^2 +2Mn\max(1, p-1)\\
		&\leq 128n\abs{p-2}C_\beta  \abs{x_1-y_1}^{\beta/2-1}+2nM\max(1, p-1)\\
		&\quad +128n\abs{p-2}C_\beta\gamma(\gamma-1)\kappa_0\abs{x_1-y_1}^{\gamma-2}\\
		&\quad-2\min(1, p-1)  \gamma(\gamma-1)\kappa_0L\abs{x_1-y_1}^{\gamma-2} +2Mr^2.
		\end{align*}
		Taking $\gamma=1+\beta/2>1$ and choosing $L$ large enough depending on $p,n, C_\beta$, we get that
		$$ 0\leq \dfrac{-\min(1, p-1)\gamma(\gamma-1)\kappa_0}{200} L\abs{x_1-y_1}^{\gamma-2}<0,   $$
		which is  a contradiction. Hence choosing first $L$ such that 
		\begin{align*}
		0&>128n\abs{p-2}C_\beta \left(\vp'(|x_1-y_1|)  \abs{x_1-y_1}^{\beta/2-1}+ |\vp''(|x_1-y_1|)|\right)\\
		&\quad +\min(1, p-1) L \vp''(|x_1-y_1|)+2nM\max(1, p-1)+2Mr^2
		\end{align*}
		and then taking $|q|$ large enough (depending on $L$, it suffices that $|q|> 6L >\frac32 |a|$ see \eqref{koivu}), we reach a contradiction and   hence $\Phi(x,y,t)\leq 0$ for $(x,y,t)\in B_r\times B_{r}\times(-r^2,0)$. The desired result follows since for $x_0,y_0\in B_{r}$, $t_0\in (-r^2,0)$, we have $\Phi(x_0,y_0,t_0)\leq 0$, we get
		\[
		|w(x_0,t_0)-w(y_0,t_0)|\leq L\vp(|x_0-y_0|)\leq L|x_0-y_0|.\qedhere
		\]
	\end{proof}
	\subsection{Improvement of oscillation and small perturbation result}

	Notice that if $|q|$ is  large enough then the equation \eqref{modhom}  satisfied by $w$ is uniformly parabolic and the operator can no longer be  discontinuous. Indeed, taking $\nu_0$ from Lemma \ref{liphom1} and assuming that $|q|>\nu_0$, we know from Lemma \ref{liphom1} that for  $(x,t)\in Q_{r}$, the gradient  $|Dw(x,t)|$ is controlled by some constant $\tilde C$ depending only on $p,n, \norm{w}_{L^\infty(Q_1)}$ and independent of $|q|$. 	It follows that,   if $q$ satisfies
	\[|q|\geq L_0:=\max(\nu_0, 4\tilde C)\geq 4\norm{Dw}_{L^{\infty}(Q_{r})},\]  we have
	$$3\tilde{C}\leq |q|-| Dw|\leq| Dw+q|.$$
	Equation \eqref{modhom} can be rewritten as 
	$$w_t-\tr(A(x,t)D^2w(x,t))=0,$$
	where $$A(x,t)=I+(p-2)\dfrac{Dw(x,t)+q}{|Dw(x,t)+q|}\otimes\dfrac{Dw(x,t)+q}{|D w(x,t)+q|}.$$
	Since we already know that $w$ is locally of class $C^{1+\beta_1, (1+\beta_1)/2}$ (see the beginning of this section where we introduced $v(x,t):=w(x,t)+q\cdot x$ and used the result of  \cite[Theorem 1.1]{JinSil2015},), we can see   that $w$ solves a uniformly parabolic equation with H\"older coefficients.
	By standard regularity results  (see \cite[Theorem 14.10]{lieberm} and  \cite[Theorem 5.1]{ladyuralpara}), we conclude that $w\in C^{2+\gamma,1+\gamma/2}_{loc}(Q_1)$.\\
	For $r\in (\frac12, \frac34)$, considering $\tilde{w}(x,t)=\dfrac{w(rx,r^2t)}{\tilde{C}}$ where $\tilde{C}$ is given in Lemma \ref{liphom1}, we have that $\tilde{w}$ is a smooth solution  to 
	\begin{equation*}
	\partial_t \tilde w-\Delta \tilde w-(p-2) \left\langle D^2\tilde w\frac{D\tilde w+\tilde q}{\abs{D\tilde w+\tilde q}}, \frac{\tilde Dw+\tilde q}{\abs{D\tilde w+\tilde q}}\right\rangle=0\quad\text{in}\quad Q_1,
	\end{equation*}
	with $ |D\tilde{w}|\leq 1$ in $Q_1$ and $|\tilde{q}|=\left|\dfrac{rq}{\tilde{C}}\right|>4r>2$.
	
	Hence without loss of generality,  when $|q|$ is large enough,  we can work  with smooth solutions to  \eqref{modhom} satisfying $|Dw|\leq 1$ and $|q|>2$.
	
	In order to derive  a H\"older gradient estimate independent of $q$ for solutions of \eqref{modhom}, we need to consider two alternatives: either we can use an improvement of oscillation  or we can show that the solution is  close to a linear function and then use a small perturbation result (see \cite{wang_para_II}). Proofs of these auxiliary results for \eqref{modhom} are quite similar to those in \cite{JinSil2015}, but for the reader's convenience we decided to give the details.
	
	First we start with an improvement of oscillation for the projection of  $Dw$ on an arbitrary direction.
	\begin{lemma} \label{firstosc}
		Let $w$ be  a smooth  solution of \eqref{modhom} with $|D w|\leq 1$ in $Q_1$,  and $|q|>2$.  For every $\ell \in (0,1), \mu>0$, there exist $\tau(\mu,n)>0, \delta(n,p,\mu)>0$, such that for any $e\in \mathbb{S}^{n-1}$, if 
		\[
		\begin{split}
		\abs{\{ (x,t)\in Q_1 \,:\,Dw(x,t)\cdot e\le 1-\ell \}}> \mu \abs{Q_1},
		\end{split}
		\]
		then 
		\[
		\begin{split}
		Dw\cdot e\le 1-\delta\text{ in }Q_\tau=B_\tau(0)\times (-\tau^2,0].
		\end{split}
		\]
	\end{lemma}
	
	\begin{proof}
		 Notice that the equation \eqref{modhom} can be rewritten as 
		\begin{equation*}
		w_t-\sum_{ij} A_{ij}(Dw) \cdot w_{ij}=0,
		\end{equation*}	
		where $$A_{ij}(s)=\delta_{ij}+(p-2)\dfrac{(s_i+q_i)}{|s+q|}\dfrac{(s_j+q_j)}{|s+q|}$$
		satisfies
		$$\min(1, p-1)I\leq A\leq \max(1, p-1)I,$$
		that is $w$ is a solution of a uniformly parabolic equation with bounded coefficients.\\
Moreover, notice that
			\begin{align*}
		A_{ijm}(s)&:= \dfrac{\partial A_{ij}(s)}{\partial s_m}\\
		&=(p-2)\dfrac{\delta_{im}(s_j+q_j)+\delta_{mj}(s_i+q_i)}{|s+q|^2}\\
		&\quad-\dfrac{2(p-2) (s_m+q_m)(s_i+q_i)(s_j+q_j)}{|s+q|^4}.
		\end{align*}
 Since by assumption $|q|>2$ and $|Dw|\leq 1$, that is $\left|Dw+q\right|\geq |q|-|Dw|\geq 1$, we observe that 
		\begin{equation}\label{filosin1}|A_{ijm}(Dw)|\leq\dfrac{4|p-2|}{\left|Dw+q\right|}\leq 4|p-2|.
		\end{equation}
		The proof  proceeds as in \cite[Lemma 4.1]{JinSil2015} by deriving the  equation, constructing good barriers and using a weak Harnack inequality \cite[Proposition 2.3]{JinSil2015} for nonnegative  supersolutions of  uniformly parabolic equations with bounded coefficients.
		 Indeed, for $\rho>0$, defining $h(x,t):=Dw(x,t)\cdot e-l+\rho|Dw(x,t)|^2$, differentiating the equation and using \eqref{filosin1}, we get that, for $c_1$ appropriately chosen, the function
$\bar h(x,t):=\frac{1-e^{-c_1((1-l+\rho)-h(x,t))}}{c_1}$
 is a nonnegative  supersolution of a uniformly parabolic equation with bounded coefficients.
		\qedhere
		
	\end{proof}
	
	If we can iterate the previous lemma to all directions, then
	proceeding by iteration,    we derive an improvement of oscillation for $|Dw|$.
	\begin{corollary}\label{iterlem}
		Assume that $|q|>2$ and let $w$ be a smooth solution to \eqref{modhom} with $|Dw|\leq 1$. For every $l\in (0,1)$ and $\mu >0$,  there exist $\tau=\tau(\mu,n)\in(0,1/4)$  and $\delta=\delta(\mu,n, p)>0$, such that for every $k\in \mathbb{N}$, if
		\begin{eqnarray*}
		a) \quad	\left|\left\{(x,t)\in Q_{\tau^i} : Dw(x,t)\cdot e\leq l(1-\delta)^i\right\}\right|>\mu |Q_{\tau^i}|\\
		\text{for all}\quad  e\in \mathbb{S}^{n-1}\quad\text{and}\quad i=0,...,k
		\end{eqnarray*}
		then 
		$$b) \quad|Dw|\leq (1-\delta)^{i+1} \quad\text{in}\quad Q_{\tau^{i+1}}\quad\text{for}\quad i=0,...,k.$$
	\end{corollary}
	
	\begin{proof}
		We proceed by induction on $k$.
		The case $k=0$ follows from the previous lemma. 
		Assume that if the assumption $a)$ holds true for $i=0,\ldots,k-1$ then the conclusion  $b)$ follows for $i=0,\ldots,k-1$.  Then it is enough to  show   that  if  assumption $a)$  is satisfied  for  $i=0,\ldots,k$ then claim  $b)$ follows for  $i=0,\ldots,k.$. Hence,  knowing that the assumption  $a)$  is satisfied for  $i=0,\ldots,k$ and  claim $b)$ holds for  $i=0,\ldots,k-1$,  it remains to show that claim  $b)$ holds for $i=k$.\\
		For this purpose,   we define the rescaled function
		$v(x,t)=\dfrac{w(\tau^k x, \tau^{2k} t)}{\tau^k(1-\delta)^k}$.
		From the induction hypothesis, we have that 
		\begin{eqnarray*}
			&|Dv|\leq 1\qquad\text{in}\quad Q_1,\\
			&|\left\{(x,t)\in Q_1 :  Dv(x,t)\cdot e\leq l\right\}|>\mu |Q_1|.
		\end{eqnarray*}
		Moreover $v$ solves in $Q_1$
		$$v_t-\Delta v-(p-2)\left\langle D^2v \dfrac{Dv(1-\delta)^k+q}{\left|Dv(1-\delta)^k+ q\right|},\dfrac{Dv(1-\delta)^k+ q}{\left|Dv(1-\delta)^k+ q\right|}\right\rangle=0.$$
		That is $v$ solves 
		$$v_t-\Delta v-(p-2)\left\langle D^2v \dfrac{Dv+\bar q}{\left|Dv+\bar q\right|},\dfrac{Dv+\bar q}{\left|Dv+\bar q\right|}\right\rangle=0$$
		where $\left|\bar q\right|=\left|\dfrac{q}{(1-\delta)^k}\right|>|q|>2$.\\
		It follows from Lemma \eqref{firstosc}, that 
		$$Dv\cdot e\leq (1-\delta)\quad \text{in}\quad Q_\tau \quad\text{for all}\quad e\in \mathbb{S}^{n-1}.$$
		This implies that $|Dv|\leq 1-\delta$ in $Q_{\tau}$ and consequently 
		\begin{equation*} |Dw|\leq(1-\delta)^{k+1}\qquad\text{in}\quad Q_{\tau^{k+1}}.\qquad \qedhere
		\end{equation*}\end{proof}

	Next, we show that under some conditions, $w$ can be arbitrary  close to a linear function. 
	\begin{lemma}\label{osclem}
		Fix a constant  $\eta >0$ and let  $w$ be a smooth solution to the uniformly parabolic equation  \eqref{modhom} with $\abs{Dw}\le 1$, $|q|>2$ and $w(0,0)=0$. Suppose that 
		\[
		\begin{split}
		\abs{\{ (x,t)\in Q_1 \,:\, \abs{Dw(x,t) -e}>\eps_{0}\}}\le \eps_{1} \text{ for some }e\in \mathbb{S}^{n-1}.
		\end{split}
		\] 
		Then if $\eps_0,\eps_1\ge 0$ are small enough, it follows that 
		\[
		\begin{split}
		\underset{(x,t)\in Q_{1/2}}{\osc}(w(x,t)-e\cdot x)\le \eta.
		\end{split}
		\]
		The smallness of  $\eps_0,\eps_1$  depends on $n$, $\eta$ and $p$.
	\end{lemma}
	
	\begin{proof}  The proof follows as in \cite[Lemma 4.4]{JinSil2015}  by first controlling  the oscillation of $w(x,t)-e\cdot x$ on  $B_{1/2}$ and using  Lemma \ref{lem:space2time}.
	\end{proof}

	Next we state a simple calculus fact, which links the assumptions of Lemma \ref{iterlem} and Lemma  \ref{osclem}. The proof is provided in \cite[pages 14-15]{JinSil2015}.
	\begin{lemma}
		\label{lem:line-of-calculus}
		Let $\ell \in (0,1)$ and $\mu>0$. Let $v:Q_1\rightarrow\R$ be a smooth function satisfying $\abs{Dv}\le 1$ in $Q_1$ and
		$$
		\abs{\{ (x,t)\in  Q_1 \,:\,Dv(x,t)\cdot e\le\ell  \}}\le \mu \abs{Q_1}
		$$
		for some $e\in \mathbb{S}^{n-1}$, then for $\eps_0:=\sqrt{2(1-\ell)}$ and $\eps_1:=\mu \abs{Q_1}$
		$$
		\abs{\{ (x,t)\in  Q_1 \,:\,\abs{Dv(x,t)-e}> \eps_0  \}}\le \eps_1.
		$$ 
\end{lemma}

	We will also use the following regularity estimate for small perturbation of solutions
	of fully nonlinear parabolic equations, which was proved by Wang \cite{wang1} and  is a parabolic analogue to the result of \cite{savin_perturb_elliptic}.
	\begin{theorem}
		\label{thm:close-to-plane}
		Let $w$ be a smooth solution to (\ref{modhom}) in $Q_1$ with $|q|>2$ and $|Dw|\leq 1$. Then for each $\gamma>0$, there are $\eta(n,p,\gamma),C(n,p,\gamma)>0$ such that
		if for some linear function $L(x)$ with $\frac12\le \abs{DL}\le 2$ it holds 
		\[
		\begin{split}
	\norm{w-L}_{L^{\infty}(Q_1)}\le \eta 
		\end{split}
		\]
		then
		\[
	\begin{split}
	\norm{w-L}_{C^{2+\gamma,1+\gamma/2}(Q_{1/2})}\le C. 
	\end{split}
	\]
	\end{theorem}
	\subsection{Proof of Proposition \ref{regmodhom}}
	We start with the case $|q|$ large enough.
	\begin{theorem} Let $w$ be a bounded viscosity solution of \eqref{modhom} with $\underset{Q_1}{\osc}\, w\leq 1$. Then there exist $L_0=L_0 (p,n, \norm{w}_{L^\infty(Q_1)})>0$, $\alpha=\alpha(p,n)\in (0,1)$ and $C=C(p,n)>0$ such that if $|q|>L_0$, then 
		$$\norm{w}_{C^{1+\alpha, (1+\alpha)/2}(Q_{1/2})}\leq C\norm{w}_{L^{\infty}(Q_1)}.$$	
	\end{theorem}
	\begin{proof}
		Let $L_0=\max(\nu_0, 4\tilde C)$ where $\nu_0$ and $\tilde C$ are given in Lemma \ref{liphom1}. If $|q|\geq L_0$,  then viscosity solutions $w$ to \eqref{modhom} are smooth solutions. Without loss of generality we can assume that $w(0,0)=0$, $|Dw|\leq 1$ and $|q|>2$.
		Let $\tau$ and $\delta$ be the constants given by Corollary \ref{iterlem}.
		Let $\eta>0$ be as in Theorem \ref{thm:close-to-plane} for a fixed  $\gamma\in (0, 1)$. 
		Let $\ell\in (0,1) $ be sufficiently close to 1 and $\mu>0$ small so that
		  $\eps_0,\eps_1>0$ defined by 
		    $$\eps_0:=\sqrt{2(1-\ell)}, \qquad \eps_1:=\mu \abs{Q_1}$$
		   are  sufficiently small in order that  the  result of  Lemma \ref{lem:line-of-calculus} holds.

		Let $k$    be the minimum of the nonnegative integers such that the condition $a)$ in Corollary \ref{iterlem} does not hold. The proof splits into two cases.\\
	\textbf{ Case 1) $k=\infty:$}\\
			In this case we have $Dw(0,0)=0$. There exists $i\in\mathbb{N}$ large enough so that, we have $\tau^{i+1}\leq \abs{x}+\sqrt{|t|}\leq\tau^{i}$, and $$i\leq \frac{\log(\abs{x}+\sqrt{|t|})}{\log(\tau)}\leq (i+1).$$
			Hence, by Corollary \ref{iterlem} it follows that
			\[
			\begin{split}
			\abs{Dw(x,t)}&\le \frac1{1-\delta}(1-\delta)^{i}\leq C (1-\delta)^{\frac{\log(\abs{x}+\sqrt{|t|})}{\log(\tau)}}\\
			&=C(1-\delta)^{\frac{\log(\abs{x}+\sqrt{|t|})}{\log(1-\delta)}\frac{\log(1-\delta)}{\log(\tau)}}\\
			&\leq C(\abs{x}+\sqrt{|t|})^\a,
			\end{split}
			\]
			where $\a=\frac{\log(1-\delta)}{\log(\tau)}$ and $C=(1-\delta)$. This proves the claim in this  case.\\
		\textbf{Case 2)} $k<\infty$:\\
			In this case we have
			\[
			\begin{split}
			\abs{\{ (x,t)\in Q_{\tau^k} \,:\, Dw(x,t) \cdot e \le \ell(1-\delta)^k \}}\le \mu \abs{Q_{\tau^k}}
			\end{split}
			\]
			for some $e\in \mathbb{S}^{n-1}$. Set
			\[
			\begin{split}
			v(x,t)=\frac1{\tau^k(1-\delta)^k}w(\tau^k x,\tau^{2k}t)
			\end{split}
			\]
			which satisfies
			$$v_t-\Delta v-(p-2)\left\langle D^2v \dfrac{Dv+\bar q}{\left|Dv+\bar q\right|},\dfrac{Dv+\bar q}{\left|Dv+\bar q\right|}\right\rangle=0\quad\text{in}\,\, Q_1$$
			where $\left|\bar q\right|=\left|\dfrac{q}{(1-\delta)^k}\right|>|q|>2$.\\
			Moreover $$
			|\left\{(x,t) :  Dv(x,t)\cdot e\leq l\right\}|\leq\mu |Q_1| \quad\text{for some}\quad e\in\mathbb{S}^{n-1},$$
			and  since condition $a)$ holds for $k-1$, we  have  $|Dv|\leq 1$ in $Q_1$.\\
			It follows from Lemma \ref{lem:line-of-calculus}, that  for $\eps_0:=\sqrt{2(1-\ell)}$ and $\eps_1:=\mu\abs{Q_1}$, we have
			$$
			\abs{\{(x,t)\in Q_1  \,:\, \abs {Dv(x,t)-e}>\eps_0 \}}\le \eps_1,$$
			and hence  by Lemma \ref{osclem}, we conclude that 
			\[ \abs{v(x,t)-e\cdot x}\le \eta \text{ for all }(x,t)\in Q_{1/2},
			\]
			where $\eta>0$ can be made arbitrarily small by choosing $\mu>0$ small and $\ell\in(0,1)$ close to one. 	By Theorem \ref{thm:close-to-plane}, there exists $C=C(p,n)>0$ such that
			$$\abs{Dv(x,t)-e}\le C(\abs{x}+\sqrt{|t|})^\gamma\quad \text{in}\,\,Q_{1/4}$$
			and since $|Dv|\leq 1$, this estimate also holds in $Q_1$.\\
			Rescaling back and assuming that $(1-\delta)/\tau\le 1$, we have
			\[
			\begin{split}
			\abs{Dw(x,t)-(1-\delta)^k e}&\le C(1-\delta)^k(\abs{x}\tau^{-k}+\sqrt{|t|}\tau^{-k})^\gamma\\
			&\le C(\abs {x}+\sqrt{|t|})^\gamma
			\end{split}
			\]
			for 	$(x,t)\in \,\,Q_{\tau^k}$.

\noindent This implies the result near zero. Standard translation and scaling arguments imply that 
		\begin{equation*}
		\norm{Dw}_{C^{\gamma, \gamma/2}(Q_{1/2})}\leq C(p,n)\norm{w}_{L^{\infty}(Q_1)}.\qedhere
		\end{equation*}
	\end{proof}
	The $C^{\frac{1+\gamma}{2}}$ H\"older continuity in time for $w$ follows as in the proof of  Lemma \ref{timereg}.
	The proof of Proposition \ref{regmodhom} is now complete since for $|q|\leq L_0$, the result follows from  the regularity estimate  from \cite{JinSil2015} for $v(x,t)=w(x,t)+q\cdot x$ (see the beginning of this section).

	\appendix
	
	\section{Characterizations of functions with H\"older continuous gradient}\label{appen1}
	\begin{lemma}\label{lemendi}
		Let $u\in C(Q_{2R}(0,0))$ for some $R>0$ and assume that there are some positive constants $c_0$ and $\alpha$ with $\alpha\leq 1$ such that for any $(x_0, t_0)$ in $Q_R(0,0)$  there is a vector $q(x_0, t_0)$ satisfying for  any $r\in (0, R]$,
		\begin{align}
		\underset{\overline{Q_r}(x_0, t_0)}{\sup}&|u(x, t)-u(x_0, t_0)-q(x_0, t_0)\cdot (x-x_0)|\nonumber\\
		&\leq \underset{Q_r(x_0, t_0)}{\osc}(u(x,t)-q(x_0,t_0)\cdot x)\nonumber\\
		&\leq c_0r^{1+\alpha}.\label{farouhat}
		\end{align}
		Then $u$ is differentiable with respect to the space variable in $Q_R(0,0)$, $Du\in C^{ \alpha, \frac\alpha2}\left(Q_{\frac R2}(0,0)\right)$ and 
		\begin{equation}\label{liebdet}
		[Du]_{C^{\alpha, \frac\alpha2}\left(Q_{\frac R2}(0,0)\right)}\leq C(n)c_0.
		\end{equation}
	\end{lemma}
	We  give details of the proof  for the reader's convenience, the result is well known.
	There are two ways to prove Lemma \ref{lemendi}. A first  idea of the proof can be found in \cite[Lemma 12.12]{lieberm}. In this section we  decided to adapt to the parabolic setting the arguments  used in \cite{krylovholder, arauijoz}. 
	\begin{proof}
	First notice that the estimate  \eqref{farouhat} implies that $u$ is differentiable with respect to the space variable in $Q_R(0,0)$ and  for $(x_0, t_0)$ in $Q_R(0,0)$, we have 	$$Du(x_0,t_0)=q(x_0,t_0).$$
	
	Let $(x,t)$, $(y,s)\in Q_{\frac R 2}(0,0)$ and define $r$ as  $r^2=|x-y|^2+|t-s|$.
	Without loss of generality we can assume that $x=-\frac{r}{4} e_1$ and $y=\frac r4 e_1$ where $e_i$ are the vector of the canonical base of $\R^n$.
	
Using \eqref{farouhat}, we have 
\begin{align}\label{pertha1}
|u(y,s)- Du(x,t)\cdot (y-x)-u(x,t)|\leq c_0r^{1+\alpha},\\
|u(x,t)-Du(y,s)\cdot (x-y)-u(y,s)|\leq c_0r^{1+\alpha}.\nonumber
\end{align}
Adding these two inequalities we get 
$$|(Du(x,t)-Du(y,s))\cdot (x-y)|\leq 2c_0r^{1+\alpha},$$
that is 
\begin{equation}\label{pertha2}
|\partial_1 u(y,s)-\partial_1 u(x,t)|\leq 4c_0r^{\alpha}.
\end{equation}
For $i=2,...,n$, we  fix $z=\frac r4 e_i$. Using \eqref{farouhat}, we have 
\begin{align*}
|u(z,t)- Du(x,t)\cdot (z-x)-u(x,t)|\leq c_0r^{1+\alpha},\\
|u(y,s)+Du(y,s)\cdot (z-y)-u(z,t)|\leq c_0r^{1+\alpha}.
\end{align*}
Adding the two inequalities  and using the triangle inequality, we get 
\begin{align*}
|u(y,s)-u(x,t)-Du(x,t)&\cdot (z-x)+Du(y,s)\cdot (z-y)|\\
&\leq 2c_0r^{1+\alpha},
\end{align*}
that is
\begin{align}\label{pertha3}
|u(y,s)-u(x,t)-Du(x,t)&\cdot (y-x)+(Du(y,s)-Du(x,t))\cdot (z-y)|\nonumber\\
&\leq 2c_0r^{1+\alpha}.
\end{align}
By the definition of $z$, we have
\begin{align*}
(Du(y,s)-Du(x,t))\cdot (z-y)&=\frac r4(\partial_i u(y,s)-\partial_iu(x,t))\\
&\quad-\frac r4(\partial_1 u(y,s)-\partial_1 u(x,t)).
\end{align*}
Using \eqref{pertha1}--\eqref{pertha3}, it follows that 
\begin{align*}
&\frac r4|\partial_i u(y,s)-\partial_iu(x,t)|\\
&\leq \frac r4|\partial_1 u(y,s)-\partial_1 u(x,t)|\\
&+|u(y,s)-u(x,t)-Du(x,t)\cdot (y-x)+(Du(y,s)-Du(x,t))\cdot (z-y)|\\
&+|u(y,s)-u(x,t)-Du(x,t)\cdot (y-x)|\\
&\leq 6c_0r^{1+\alpha}.\qedhere
\end{align*}
\end{proof}

\end{document}